\documentclass[11pt,a4paper,openright,oneside]{article}
\usepackage{amsfonts, amsmath, amssymb,latexsym,amsthm, mathrsfs, enumerate,mathtools}
\usepackage[english]{babel}
\usepackage{epsfig}
\usepackage{xcolor}
\usepackage{cite}
\usepackage[margin=1.2in]{geometry}
\usepackage{graphicx}
\usepackage{listings}
\usepackage[latin1]{inputenc}

\newcommand{\dotcup}{\,\ensuremath{\mathaccent\cdot\cup}\,}

\newcommand{\nloop}{\mathrlap{\circlearrowleft}\,\setminus}

\makeatletter
\providecommand\dotbigcup{\mathpalette\@barred\cdot}
\def\@barred#1#2{\ooalign{\hfil$#1\bigcup$\hfil\cr\hfil$#1#2$\hfil\cr}}

\makeatother

\numberwithin{equation}{section}
\newtheorem{teo}{Theorem}[section]
\newtheorem*{teo*}{Theorem}
\newtheorem*{prop*}{Proposition}
\newtheorem*{corol*}{Corollary}
\newtheorem{prop}[teo]{Proposition}
\newtheorem{corol}[teo]{Corollary}
\newtheorem{lema}[teo]{Lemma}
\newtheorem{lemma}[teo]{Lemma}

\theoremstyle{definition}

\newtheorem*{claim}{Claim}

\newtheorem{obs}[teo]{Observation}

\newtheorem{remark}[teo]{Remark}
\newtheorem{quest}[teo]{Question}
\newtheorem{theorem}[teo]{Theorem}

\newcommand{\Cay}{\mathrm{Cay}}
\newcommand{\End}{\mathrm{End}}
\newcommand{\Aut}{\mathrm{Aut}}

\usepackage{fancyhdr}

\begin{document}

\title{On monoid graphs}
\author{Kolja Knauer \and Gil Puig i Surroca}

\maketitle

\begin{abstract}
 We investigate Cayley graphs of finite semigroups and monoids. First, we look at semigroup digraphs, i.e., directed Cayley graphs of semigroups, and give a Sabidussi-type characterization in the case of monoids. We then correct  a proof of Zelinka from '81 that characterizes semigroup digraphs with outdegree $1$. Further, answering a question of Knauer and Knauer, we construct for every $k\geq 2$  connected $k$-outregular non-semigroup digraphs. On the other hand, we show that every sink-free directed graph is a union of connected components of a monoid digraph.
 
 Second, we consider monoid graphs, i.e., underlying simple undirected graphs of Cayley graphs of monoids. We show that forests and threshold graphs form part of this family. Conversely, we construct the -- to our knowledge -- first graphs, that are not monoid graphs. We present non-monoid graphs that are planar, have arboricity $2$, and treewidth $3$  on the one hand, and non-monoid graphs of arbitrarily high connectivity on the other hand. 
 
 Third, we study generated monoid trees, i.e., trees that are monoid graphs with respect to a generating set. We give necessary and sufficient conditions for a tree to be in this family, allowing us to find large classes of trees inside and outside the family.
\end{abstract}
%
%
%

\section{Introduction}

After being introduced in 1878~\cite{Cayley1878}, Cayley graphs soon became a prominent tool in both graph and group theory. As of today Cayley graphs of groups play a prominent role in (books devoted to) algebraic graph theory, see e.g.~\cite{God-01,Kna-19}. On one hand, they are useful for constructions. They are crucial in the proof of Frucht's theorem~\cite{Frucht1939} and the related representation problems~\cite{Babai1981}, i.e., how to represent a group as the automorphism group of a graph. On the other hand, Cayley graphs endow groups with metric or topological properties. For instance, White defined the genus of a group as the minimum genus for any connected Cayley graph of the group~\cite{White1972}. This  created intricate links to symmetry groups of surfaces, see~\cite{GrossTucker1987,White1984}. 
 
Cayley graphs of monoids and semigroups (usually considered as \emph{directed Cayley graphs}) have been studied less, but still there is a considerable amount of work, see e.g.~\cite{Kna-19} for a book and the references therein and~\cite{KRY2009} for a survey on this subject. Characterizations of Cayley graphs of certain classes of semigroups have been subject to some research effort, see~\cite{zbMATH05610940,Kel-06,zbMATH06864652,zbMATH06948232,zbMATH06613956,zbMATH06184562,zbMATH06147894,zbMATH06029728,zbMATH06093204,zbMATH05973394,zbMATH05886836,zbMATH06139402,zbMATH05701935,zbMATH05812671}.
Also Cayley graphs from certain restricted families have been studied, such as transitively acyclic, directed ones also known as Cayley posets~\cite{gar-20} as well as generalized Petersen graphs~\cite{zbMATH05973394,zbMATH06093204,Lov-97,GMK21} and others~\cite{zbMATH06603377,zbMATH06120589}.
  
Along the lines of White's genus of a group characterizations of semigroups admitting Cayley graphs with topological embeddability properties have been investigated extensively~\cite{Sol-06,Sol-11,Zha-08,Kna-10,Kna-16}. Note that for such topological questions edge directions, multiplicities and loops can usually be suppressed, which leads to the notion of \emph{underlying simple undirected Cayley graph}. Also other questions of the type which semigroups admit a Cayley graph from a given class have been considered~\cite{zbMATH05811972,zbMATH05647985,Kel-02,Kel-03}.
This latter type of question is related to the representation problem for monoids. Indeed, a standard strategy to represent a monoid is to take its \emph{colored Cayley graph}, see Figure~\ref{fig:ExampleMonoidGraph2}, and mimic the colors with certain rigid gadgets on the edges. Then the monoid is isomorphic to the endomorphism monoid of the resulting graph, see~\cite{HL69,HP64,HP65}. Thus, if we know that a monoid has a Cayley graph in a given class, then (depending on the edge gadgets) it is the endomorphism monoid of a graph from that class. There are classic and recent questions and results on how sparse a class can be be while being rich enough to represent all groups~\cite{Frucht1939,grohe2020automorphism,B74} or monoids, see~\cite[Problem 19.2]{NOdM12} or~\cite{BP80}.
  
\begin{figure}[h] 
\centering
\includegraphics[width=.8\textwidth]{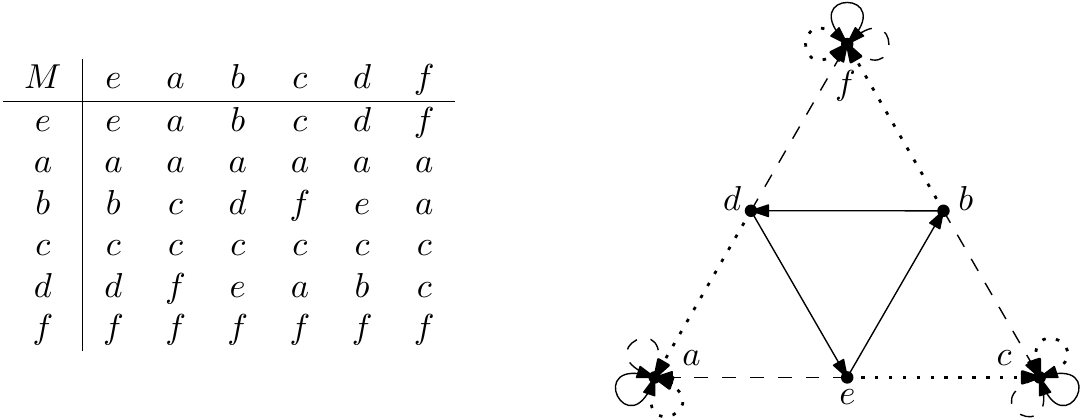}
\caption{Left: the multiplication table of a monoid $M$ with neutral element $e$. Right: the colored Cayley graph $\Cay_{\mathrm{col}}(M,\{a,b,c\})$ with connection set $\{a,b,c\}$.}\label{fig:ExampleMonoidGraph2}
\end{figure}

In this paper we study the related question of how large the class of (finite) Cayley graphs is and if some structural properties can be shown similar to the case of Cayley graphs of groups. In particular, a starting point was the question whether sparsity can ensure Cayleyness. A similar study has been carried out with respect to Cayley posets in~\cite{gar-20}. 

First, we study the class of directed Cayley graphs of semigroups and monoids called \emph{semigroup digraphs} and \emph{monoid digraphs}, respectively.
In Section~\ref{sec:basic_properties}, we present an easy but to our knowledge unpublished analogue of Sabidussi's Lemma~\cite{Sabidussi1958} for monoid digraphs. This answers a question of~\cite[Chapter 11.3]{Kna-19}. Section~\ref{sec:digraphs} further studies the class of semigroup and monoid digraphs. We review results of Zelinka~\cite{Zelinka1981}, who characterized such digraphs with outdegree $1$. In particular, we correct his proof. Answering a question of Knauer and Knauer~\cite[below Proposition 11.3.2]{Kna-19}, we construct for every $k\geq 2$ connected $k$-outregular digraphs (of high (strong) connectivity) that are not semigroup digraphs. On the other hand, we show that every sink-free directed graph, i.e.~with minimum outdegree at least $1$,  is a union of connected components of a monoid digraph. 
This can be seen as an analogue to the result that every graph is isomorphic to some induced subgraph of a Cayley graph of a group~\cite{Babai1978}.

Second, we focus on the underlying simple undirected graphs of Cayley graphs of monoids, called \emph{monoid graphs}. As mentioned earlier removing directions, orientations, and loops is natural when considering embeddability and sparsity questions of the resulting graph and has prompted questions about this class, see e.g.~\cite{GMK21,Kna-10,Kna-16}. On the other hand, one difficulty that arises in their study is that they carry less categorical and algebraic structure than their directed non-simple counterparts, see e.g.~\cite[Chapter 11.1]{Kna-19}. In Section~\ref{section:monoid_graphs} we show that forests and threshold graphs are monoid graphs. After that, we show to our knowledge the first graphs, that are not monoid graphs. These results were the starting point for this paper.  In particular, we present families of planar non-monoid graphs of arboricity $2$ and treewidth $3$ and construct non-monoid graphs of arbitrarily high connectivity based on Ramanujan graphs. 
 
Finally, we study monoid graphs where the connection set of the Cayley graph generates the monoid. In Section~\ref{sec:monoid_tree} we give a necessary and a sufficient condition for a tree to be in that class. This allows us to find large families of trees inside and outside of the family. 

There are still many intriguing open questions, which we survey in Section~\ref{sec:conclusion}.
%
%


\section{Preliminaries, properties, and characterizations}\label{sec:basic_properties}

For a semigroup $S$ and a subset $C\subseteq S$ called \emph{connection set}, we consider three different Cayley graphs. The directed graph with vertex set $S$ and arc set $\{(s,sc)\ |\ s\in S,\,c\in C\}$, denoted $\Cay(S,C)$, is a digraph with possible loops and anti-parallel arcs, but without parallel arcs. We call the resulting digraphs \emph{semigroup} and \emph{monoid digraphs}, respectively. Note that the construction allowing parallel arcs has been studied from a categorical point of view, see e.g.~\cite[Chapter 11.1]{Kna-19}. Suppressing multiplicities still yields a faithful covariant functor as in the case with parallel arcs.

The underlying simple undirected graph of $\Cay(S,C)$, which we denote $\underline{\Cay}(S,C)$, carries less information. It looses categorical and algebraic properties, but as justified above, is natural from a graph theoretical point of view. We call the resulting graphs \emph{semigroup} and \emph{monoid graphs}. In the last section we will further restrict this class by requiring that the connection set $C$ is a generating set, yielding the notion of \emph{generated monoid graphs}.

The richest representation of the pair $(S,C)$ is the multi-digraph with colored arcs $\Cay_{\mathrm{col}}(S,C)$, where the colors correspond to the elements of $C$ generating the arcs. Here for each $c$ there is an arc $(s,sc)$ of color $c$, in particular, we allow multiple parallel arcs.
%
%


We denote by $\End(G)$ and $\Aut(G)$ the endomorphism monoid and automorphism group of $G$, respectively. Here $G$ can be colored and directed, only directed, or 
undirected. In directed and undirected graphs endomorphisms just have to map arcs to arcs and edges to edges, respectively. If $G$ is colored and directed then an endomorphism has to map any arc to an arc of the same color. Automorphisms are just bijective endomorphisms. Clearly, the endomorphism monoid of a colored digraph is a submonoid of the endomorphism monoid of the digraph without colors. However, suppressing loops may remove some endomorphisms. This is, for a semigroup $S$ and $C\subseteq S$, we have $\End(\Cay_{\mathrm{col}}(S,C))\leq\End(\Cay(S,C))$ but there might be $\varphi\in \End(\underline{\Cay}(S,C))\setminus\End(\Cay(S,C))$. 

The multiplication law of a semigroup yields information shared by all the (colored) Cayley digraphs it defines. See the following basic lemma.

\begin{lemma}\label{lem:leftM} If $S$ is a semigroup and $C\subseteq S$, then mapping every $s\in S$ to left-multiplication $\varphi_s$ with $s$ is a homomorphism from $S$ to $\End(\Cay_{\mathrm{col}}(S,C))$.  
\end{lemma}
\begin{proof}
If $s,t\in S$, $c\in C$ and $(t,tc)$ is an arc of 
$\Cay_{\mathrm{col}}(S,C)$ having color $c$, then $(st,stc)$ is also an arc of $\Cay_{\mathrm{col}}(S,C)$ having color $c$. Thus, left-multiplication $\varphi_s$ with $s$ is an element of $\End(\Cay_{\mathrm{col}}(S,C))$. Clearly, $\varphi_{st}=\varphi_{s}\circ \varphi_{t}$, so this is a semigroup homomorphism.
\end{proof}

A well-known strengthening of Lemma~\ref{lem:leftM} is the following property of \emph{generated} colored Cayley graphs of monoids, see e.g.~\cite[Theorem 7.3.7]{Kna-19}:

\begin{lemma}\label{lem:colM}
Let $M$ be a monoid and be $C\subseteq M$ such that $M=\langle C\rangle$. Then left-multiplication yields an isomorphism from $M$ to $\Cay_{\mathrm{col}}(M,C)$. 
\end{lemma}
 
Results that complement the above lemmas with a converse direction are sometimes called \emph{Sabidussi-type} characterizations. They rely on  certain properties of the automorphism group or endomorphism monoid. The first such result is due to  Sabidussi~\cite[Lemma 4]{Sabidussi1958}: 

\begin{lemma}\label{lem:Sabidussi}
A (colored, directed) graph $G$ is the (colored, directed) Cayley graph of a group if and only if $\Aut(G)$ has a subgroup that acts regularly on $G$.  
\end{lemma}

Here we prove a Sabidussi-type characterization of monoid digraphs, that was asked for in~\cite[Chapter 11.3]{Kna-19}.



\begin{lemma}\label{lem:monosabi} A directed graph $G=(V,A)$ is a monoid digraph if and only if there exists a vertex $e\in V$ and a submonoid $M\leq\End(G)$ such that for each $x\in V$ there is a unique $\varphi_x\in M$ with $\varphi_x(e)=x$. Moreover, $\varphi_x$ satisfies that for each $(x,y)\in A$ there is $(e,c)\in A$ such that $\varphi_x(c)=y$.
\end{lemma}
\begin{proof}
Suppose that there is a vertex $e$ and a submonoid  $M\leq\End(G)$ satisfying this property. Let $C=\{\varphi\in M\ |\ (e,\varphi(e))\in A\}$. We claim that there is an isomorphism $G\cong\Cay(M,C)$ given by $x\mapsto\varphi_x$. It is clear that this map is injective. It is also surjective: the preimage of $\varphi\in M$ is $\varphi(e)$. Now, if $(x,y)\in A$, there is an out-neighbor $c$ of $e$ with $\varphi_x\circ\varphi_c(e)=\varphi_x(c)=y$; thus $\varphi_x\circ\varphi_c=\varphi_y$. Since $\varphi_c\in C$, $(\varphi_x,\varphi_y)$ is an arc of $\Cay(M,C)$. Conversely, if $(\varphi_x,\varphi_y)$ is an arc of $\Cay(M,C)$ then $\varphi_x\circ\varphi=\varphi_y$ for some $\varphi\in C$. This implies that the arc $(e,\varphi(e))\in A$ is mapped by $\varphi_x$ to $(x,\varphi_x\circ\varphi(e))=(x,y)$, which must be also in $A$, because $\varphi_x$ is an endomorphism.

Conversely, suppose that $G=\Cay(M,C)$ for a monoid $M$ and $C\subseteq M$. For each vertex $x$ consider the endomorphism $\varphi_x$ of $G$ defined by left-multiplication by $x$. All these endomorphisms together form a submonoid $M'\leq\End(\Cay_{\mathrm{col}}(M,C))\leq\End(G)$, by Lemma~\ref{lem:leftM}. Let $e$ be the neutral element of $M$. It is clear that for any vertex $x$, the mapping $\varphi_x$ is the only of them that maps $e$ to $x$. Moreover, for any out-neighbor $y$ of $x$ there is an out-neighbor $c\in C$ of $e$ with $xc=y$, so $\varphi_x(c)=y$. 
\end{proof}

We end the section with a lemma which shows how the connectivity of a semigroup digraph yields information about its defining semigroup(s). A directed graph is \emph{(strongly) connected} if for any two vertices $x,y$ there is a (directed) path from $x$ to $y$. A set of arcs $A'\subseteq A$ of a connected directed graph $G=(V,A)$ is a \emph{directed edge cut} if $(V,A\setminus A')$ has exactly two connected components $V_1$ and $V_2$ and all arcs of $A'$ are directed from $V_1$ to $V_2$. It is an easy observation, that a connected directed graph $G=(V,A)$ is strongly connected if and only if it has no directed edge cut.  

\begin{lemma}\label{lem:strongconn} If  $G=\Cay(S,C)$ is strongly connected, then $S$ is left cancellative.
\end{lemma}
\begin{proof}
For any $s\in S$ and any $X\subseteq S$ with $XC\subseteq X$ we have that $sXC\subseteq sX$. In particular, $sS$ has no outgoing arcs. Hence if $sS\subsetneq S$, then there is a directed edge cut defined by the arcs from $S\setminus sS$ to $sS$. Thus, strong connectivity yields $sS=S$. By Lemma~\ref{lem:leftM} and the fact that $S$ is finite, left-multiplication by $s$ is an automorphism of $G$. This means that $S$ is left cancellative.
\end{proof}

\section{Semigroup digraphs and monoid digraphs}\label{sec:digraphs}

In this section we study semigroup and monoid digraphs. 
If for a directed graph $G$ there is a semigroup $S$ and a non-empty $C\subseteq S$ with $G=\Cay(S,C)$, then the outdegree of any vertex of $G$ is between $1$ and $|C|$. In the \emph{$1$-outregular} case, i.e.~when all vertices have outdegree $1$, one can reduce $C$ to any of its members without changing $G$:

\begin{obs}\label{obs:1}
 A semigroup digraph $G=\Cay(S,C)$ is $1$-outregular if and only if  $G=\Cay(S,\{a\})$ for some $a\in C$.
\end{obs}

The $1$-outregular semigroup digraphs were studied by Zelinka~\cite{Zelinka1981}. He gave a very concrete characterization of this class, but his construction of the corresponding semigroups is flawed. However, an alternative non-trivial construction can be given inspired by \cite[Proposition 11.3.2]{Kna-19}. This is the first objective of this section. We give a monoid version of the result (Theorem~\ref{Monoid_Zelinka}), and then obtain Zelinka's theorem as a corollary (Corollary~\ref{Zelinka}).

In a \emph{$1$-outregular} digraph $G$ each connected component $\mathcal C$ has exactly one cycle (possibly a loop) $Z$. All vertices in $\mathcal C$ have a unique shortest directed path to $Z$. For $v\in\mathcal C$ denote by $\ell(v)$ the length of this path and by $\ell(\mathcal C)$ the maximum among all these lengths. Moreover, denote by $z(\mathcal C)$ the length of $Z$.

\begin{theorem}\label{Monoid_Zelinka} A $1$-outregular digraph $G$ is a monoid digraph if and only if $G$ has a component $\mathcal C$ such that $z(\mathcal D)$ divides $z(\mathcal C)$ and $\ell(\mathcal D)\leq\ell(\mathcal C)$ for all components $\mathcal D$ of $G$.
\end{theorem}

\begin{proof} By Observation~\ref{obs:1}, if $G=(V,A)$ is a monoid digraph, then $G=\Cay(M,\{a\})$ for some monoid $M$ and $a\in M$. Let $k$ and $h$ be the \emph{pre-period} and the \emph{period} of $a$, respectively. This is, $e,a,...,a^{k+h-1}$ are pairwise distinct and $a^{k+h}=a^k$. Hence, the component $\mathcal C$ containing $a$ satisfies $z(\mathcal C)=h$ and $\ell(\mathcal C)\geq k$. Now let $\mathcal D$ be any component of $G$, $x\in\mathcal D$ at distance $p=\ell(\mathcal D)$ from the cycle of $\mathcal D$, and $q=z(\mathcal D)$. If $q$ does not divide $h$, then $k$ and $k+h$ are not congruent modulo $q$ and thus $xa^{k}\neq xa^{k+h}$, a contradiction. Hence $z(\mathcal D)$ divides $z(\mathcal C)$. Now suppose $p\geq k+1$. Then $xa^k$ is distinct from $xa^l$ for each $l\neq k$, but this is also a contradiction. Hence $p\leq k\leq\ell(\mathcal C)$.

Now suppose that the condition is fulfilled. For $x,y\in V$  denote by $d(x,y)$ the length of the shortest directed path from $x$ to $y$, if it exists. Moreover, for $x\in \mathcal{D}$ denote $z(x):=z(\mathcal{D})$. In $\mathcal C$ choose a vertex at distance $\ell(\mathcal C)$ from the cycle $Z\subseteq\mathcal C$. Conveniently call this vertex $e$ and call its out-neighbor $a$. For any vertex $x\in V$ and any $k\in\mathbb N$, define $x+k$ as the end vertex of the directed walk of length $k$ starting at $x$.  It is clear that for any $k,\ell\in\mathbb N$ the vertices $x+k,x+\ell$ are equal if and only if $d(x,x+k)=d(x,x+\ell)$.

\begin{claim}\label{claim}
For every $x\in V$ and every $k\in\mathbb N$ we have $d(x,x+k)=d(x,x+d(e,e+k))$.
\end{claim}
\begin{proof}
Observe that 
\[d(x,x+k)=\begin{cases}k &\text{if } k<\ell(x)\\ 
\ell(x)+((k-\ell(x))\!\!\!\mod z(x)) &\text{otherwise,}\end{cases}\]
where $(k-\ell(x))\!\mod z(x)$ is the remainder of the Euclidean division of $k-\ell(x)$ by $z(x)$. Now we distinguish two cases. If $k<\ell(e)$ then $d(e,e+k)=k$. If otherwise $k\geq\ell(e)$, write $k=\ell(e)+z(e)q_k+r_k$, where $q_k,r_k$ are the integers determined by $0\leq r_k<z(e)$. By hypothesis, $k\geq\ell(e)\geq\ell(x)$ and $z(x)\mid z(e)$, so $d(x,x+d(e,e+k))=d(x,x+\ell(e)+r_k)=d(x,x+\ell(e)+z(e)q_k+r_k)=d(x,x+k)$.
\end{proof}

Now, let $\omega=e+(\ell(\mathcal C)+z(\mathcal C)-1)$. Observe that $\omega\in Z\subseteq\mathcal C$. Therefore $\omega$ can be reached from any vertex $v$ of $\mathcal C$, so we can set $r(v):=d(e,\omega)-d(v,\omega)$. Note that we have $d(v,\omega)=d(e+r(v),\omega)$. This implies that $d(v+k,\omega)=d(e+r(v)+k,\omega)$ for any $k\in\mathbb N$, so $r(v+k)=r(e+r(v)+k)=d(e,e+r(v)+k)$ for any $k\in\mathbb N$. We regard the vertices of $G$ as the elements of a monoid $M$ equipped with the following multiplication:
\begin{align*}
\forall y\in V \ \ &ey=y; \\
\forall x\in V\!\setminus\!\{e\}\ \;\forall y\in V \ \ &xy=\begin{cases}x+r(y) &\text{if }y\in \mathcal C\\y &\text{otherwise.}\end{cases}
\end{align*}
By definition, $e$ is the right-neutral element, but also $xe=x+r(e)=x$. Hence $e$ is the neutral element of this operation. Let us check that the operation is associative. Pick any $x,y,z\in V$. If $y\notin\mathcal{C}$, then $yz\notin\mathcal{C}$ and $(xy)z=yz=x(yz)$. If $z\notin\mathcal{C}$ then $(xy)z=z=x(yz)$. So suppose  $y,z\in\mathcal C$. Then using the above Claim we can transform, $(xy)z=(x+r(y))+r(z)=x+d(x,x+r(y)+r(z))=x+d(x,x+d(e,e+r(y)+r(z)))=x+d(e,e+r(y)+r(z))=x+r(y+r(z))=x+r(yz)=x(yz)$. Finally, note that $(x,y)$ is an arc of $G$ if and only if $y=x+1=xa$. Therefore, $G=\Cay(M,\{a\})$.
\end{proof}

\begin{corol}[Zelinka's Theorem]\label{Zelinka} 
A $1$-outregular digraph $G$ is a semigroup digraph if and only $G$ has a component $\mathcal C$ such that $z(\mathcal D)$ divides $z(\mathcal C)$ and $\ell(\mathcal D)\leq\ell(\mathcal C)+1$ for all components $\mathcal D$ of $G$.
\end{corol}

\begin{proof} Again by Observation~\ref{obs:1} we can suppose that $G=\Cay(S,\{a\})$ for a semigroup $S$ and $a\in S$. Let $\mathcal C$ be the connected component of $a$ and $S'$ the monoid obtained from $S$ by adjoining a neutral element $e$. We can establish a bijective correspondence $\mathcal D\mapsto \mathcal D'$ between connected components of $G$ and $\Cay(S',\{a\})$ as follows: $\mathcal D'=\mathcal D$ if $\mathcal D\neq \mathcal C$, and $\mathcal C'$ is obtained from $\mathcal C$ by adding the vertex $e$ and the arc $(e,a)$. Observe that $z(\mathcal C')=z(\mathcal C)$ and $\ell(\mathcal C')\leq\ell(\mathcal C)+1$. By Theorem~\ref{Monoid_Zelinka}, $z(\mathcal D)\mid z(\mathcal C)$ and $\ell(\mathcal D)\leq\ell(\mathcal C)+1$ for any connected component $\mathcal D$ of $G$.

Now suppose that $G=(V,A)$ fulfills the condition. Pick a vertex $v$ in $\mathcal C$ with $\ell(v)=\ell(\mathcal C)$, and from $G$ construct a new $1$-outregular graph $G'$ by adding a new vertex $u$ and the arc $(u,v)$. By Theorem~\ref{Monoid_Zelinka}, $G'$ is isomorphic to $\Cay(M,\{a\})$ for some monoid $M$ and some $a\in M$. It follows from its proof that $u$ can be assumed to be the neutral element $e$ and $v$ to be $a$. Additionally, since $e$ has no in-neighbors, it can be checked that $M\setminus\{e\}$ is closed under the constructed product. Therefore, it is a semigroup, and $G$ is isomorphic to $\Cay(M\!\setminus\!\{e\},\{a\})$.
\end{proof}

The following has been claimed before~\cite[Proposition 11.3.2]{Kna-19}. However, also in this proof there is a slight problem in the monoid construction. Here we get it as a direct consequence of Theorem~\ref{Monoid_Zelinka}:
\begin{corol}\label{cor:connected}
 Any connected $1$-outregular digraph is a monoid digraph.
\end{corol}

%

It was asked in~\cite[below Proposition 11.3.2]{Kna-19} if Corollary~\ref{cor:connected} could be extended to $k$-outregular semigroup digraphs. Here we give a negative answer in a strong sense. A subset of vertices $V'$ of a (strongly) connected (directed) graph $G$ is called \emph{(directed) vertex cut} if $G\setminus V'$ is not (strongly) connected. If $\kappa$ is an integer such that all (directed) vertex cuts of $G$ are of order at least $\kappa$ and $\kappa+1\leq n$, where $n$ is the order of $G$, then $G$ is \emph{(strongly) $\kappa$-connected}. 

\begin{theorem}\label{nonsemigroup_family}
For every $k>1$ there exist infinitely many strongly connected $k$-outregular non-semigroup digraphs. Moreover, for every $\kappa>0$ with $\left\lfloor\frac{k}{\kappa}\right\rfloor>1$ there are such digraphs that are strongly $\kappa$-connected, and whose underlying undirected graph is $(k+\kappa)$-connected.
For $\left\lfloor\frac{k}{\kappa}\right\rfloor>2$ such digraphs exist that cannot even be turned into semigroup digraphs by adding loops.
\end{theorem}
\begin{proof}
%
%
%
%

Given two positive integers $k,\ell$, let $G_{k,\ell}$ be the directed graph with vertex set $\{(0,j)\in\mathbb Z^2\ |\ 0\leq j\leq k^2-1\}\cup\{(i,j)\in\mathbb Z^2\ |\ 1\leq i\leq\ell-1,\,0\leq j\leq k-1\}$ which has an arc from $(i_1,j_1)$ to $(i_2,j_2)$ if and only if $i_1+1=i_2$ or alternatively $i_1=\ell-1$, $i_2=0$ and $j_1=\left\lfloor \frac{j_2}{k}\right\rfloor$. 
Given a positive integer $\kappa$, let $\sim_{\kappa}$ be an equivalence relation on $\{(0,j)\in\mathbb Z^2\ |\ 0\leq j\leq k^2-1\}$ defined as $(0,j_1)\sim_{\kappa} (0,j_2)$ if and only if the following conditions are satisfied:
\begin{itemize}
\item[(i)] $j_1\equiv j_2 \mod k$;
\item[(ii)] $\left\lfloor\frac{j_1}{k\kappa}\right\rfloor=\left\lfloor\frac{j_2}{k\kappa}\right\rfloor$, or $\kappa\nmid k$ and $\left\lfloor\frac{j_1}{k\kappa}\right\rfloor+1=\left\lfloor\frac{j_2}{k\kappa}\right\rfloor=\left\lfloor\frac{k^2-1}{k\kappa}\right\rfloor$ (or the symmetric condition).
\end{itemize}

We call $G_{k,\ell,\kappa}$ the directed graph obtained from $G_{k,\ell}$ by merging all vertices of each $\sim_{\kappa}$-class. See Figure~\ref{fig:nonsemi} for an illustration.

\begin{figure}[h]
\centering
\includegraphics[width=1\textwidth]{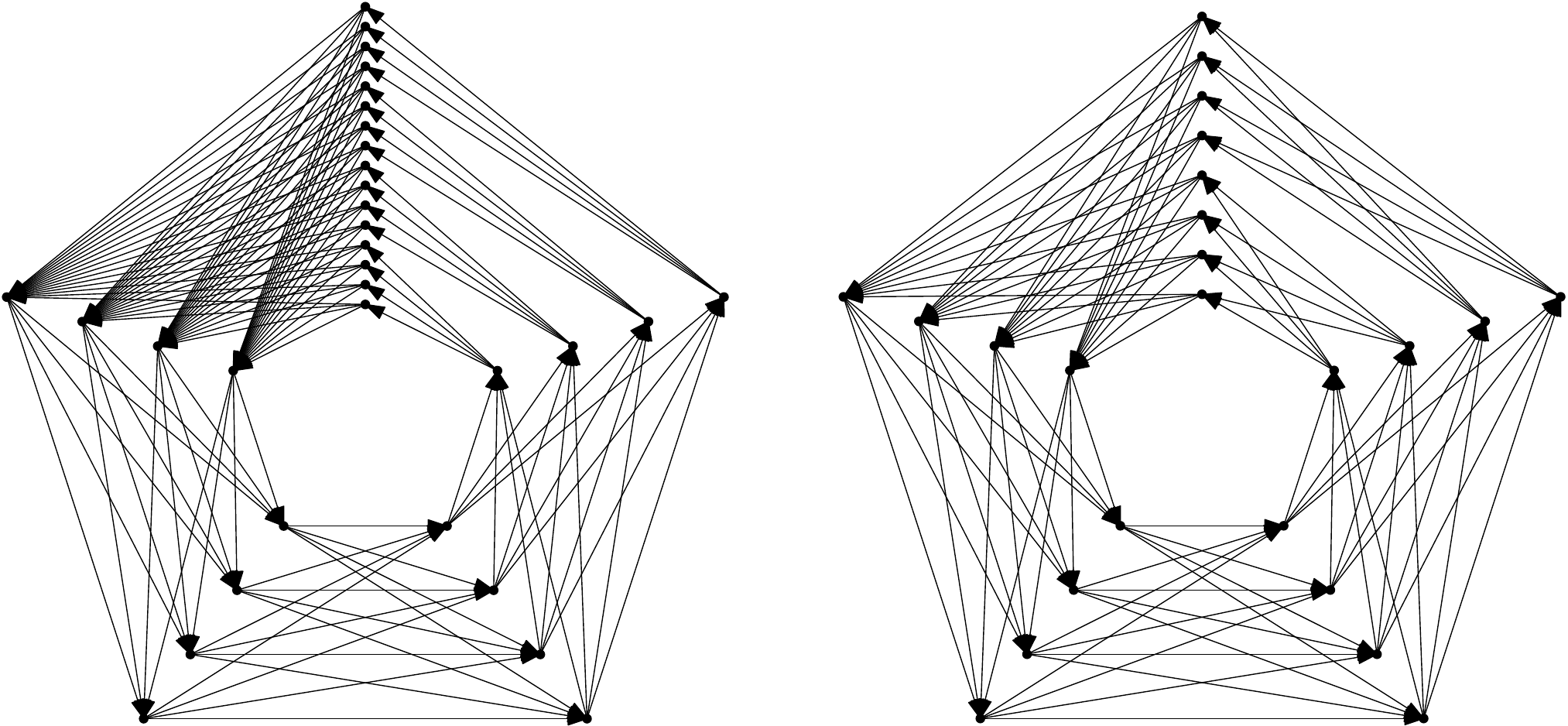}
\caption{Left: the directed graph $G_{4,5}\cong G_{4,5,1}$. Right: the merged directed graph $G_{4,5,2}$.}\label{fig:nonsemi}
\end{figure}


Note that for $\ell\geq 2$ the digraph $G_{k,\ell,\kappa}$ is $k$-outregular and strongly connected. Denote by $V_i$ the set of vertices of $G_{k,\ell,\kappa}$ with first coordinate congruent to $i$ modulo $\ell$ if $\ell\nmid i$, and the set of $\sim_{\kappa}$-classes if $\ell\mid i$. In particular, $|V_0|=\left\lfloor\frac{k}{\kappa}\right\rfloor k$ and $|V_1|=...|V_{\ell-1}|=k$. Assume that $\left\lfloor\frac{k}{\kappa}\right\rfloor>1$ and $\ell\geq 4$.

It is clear that $G_{k,\ell,\kappa}$ is not strongly $(\kappa+1)$-connected: for instance, all in-neighbors of any vertex of $V_0$ form a directed vertex cut. Now let $X$ be a minimal directed vertex cut. Suppose that there is a vertex $x\in X\cap V_i$ with $i\not\equiv -1 \mod\ell$. Then since $|X|\leq\kappa<k$ there is in $V_i\setminus X$ a vertex with the same in-neighbors and out-neighbors as $x$, so $X\setminus\{x\}$ is also a directed vertex cut, a contradiction. Therefore $X\subseteq V_{-1}$. Hence there must be a vertex in $V_0$ with all in-neighbors in $X$, or $X$ would not be a directed vertex cut. This implies that $|X|=\kappa$.

Moreover the underlying undirected graph of $G_{k,\ell,\kappa}$ is $(k+\kappa)$-connected. First note that $V_1\cup Y$, where $Y$ is the set of in-neighbors of a vertex of $V_0$, is a vertex cut. Now let $X$ be a minimal vertex cut; it is clear that $V_{i}\subseteq X$ for some $i$. This $i$ is unique (modulo $\ell$); more precisely, it is the only index satisfying $|X\cap V_i|\geq k$. Suppose that there is a vertex $x\in X\cap V_j$ with $j\not\equiv -1,i \mod \ell$. Then there is a vertex in $V_j\setminus X$ with the same neighbors as $x$, so $X\setminus\{x\}$ is also a vertex cut, a contradiction. Therefore $X\subseteq V_{-1}\cup V_i$. Since $\ell\geq 3$ it is clear that $V_{-1},V_0\neq V_i$. Moreover, since $\ell\geq 4$ there is a vertex $y\in V_0\setminus X$ with all its in-neighbors in $X$. This implies that $k+\kappa\leq|V_i|+|Y|\leq|X|\leq k+\kappa$, where $Y$ is the set of in-neighbors of $y$.

We will now show that:\begin{itemize}
                                                                                                                                                                                                                                                                                                                          \item[(a)] under the assumed hypotheses on $k,\ell,\kappa$, $G_{k,\ell,\kappa}$ is not a semigroup digraph;
\item[(b)] if additionally $\left\lfloor\frac{k}{\kappa}\right\rfloor>2+\frac{1}{k}$, then there is no semigroup $S$ and $C\subseteq S$ for which $G_{k,\ell,\kappa}=\Cay^{\nloop}(S,C)$, where $\Cay^{\nloop}(S,C)$ denotes the graph obtained from $\Cay(S,C)$ by removing all loops.                                                                                                                                                                                                                                                                                                                         \end{itemize}

In both cases a hypothetical representation with a semigroup $S$ and a connection set $C$ leads to a contradiction. Indeed, let $x\in S$ and $X\subseteq S$. By Lemma~\ref{lem:strongconn} $S$ is left cancellative, so $|xX|=|X|$. Now suppose that $x\notin V_{-2}$ and let $y\in V_{-2}$. We have $k=|xC^2|=|C^2|=|yC^2|=\left\lfloor\frac{k}{\kappa}\right\rfloor k$ if there is a semigroup representation in case (a) and $2k+1\geq|xC^2|=|C^2|=|yC^2|\geq\left\lfloor\frac{k}{\kappa}\right\rfloor k$ if there is a semigroup representation in case (b), the desired contradictions.
%
%
\end{proof}
%

On the other hand we also know a smallest outregular non-semigroup digraph.

\begin{prop}\label{nonsemigroup}
 There is an outregular non-semigroup digraph on three vertices and this is smallest possible.
\end{prop}
\begin{proof}
 Consider the directed graph $G$ from Figure~\ref{fig:notcayleysemigroup} with vertex set $\{x,y,z\}$ and arc set $$\{(x,x),(x,y),(y,x),(y,z),(z,y),(z,z)\}.$$ Suppose that $G=\Cay(S,C)$ for a semigroup $S$ and some $C\subseteq S$. 
It is clear that there is no $u\in S$ with $ux=y$ or $uz=y$. On the other hand, $\exists c,c'\in C\ \, xc=zc'=y$. Therefore, $c=c'=y\in C$. Now suppose that $y^2=x$. Then $yx=xy=y$, a contradiction. Similarly, $y^2\neq z$. Hence $y^2=y$, but this is also a contradiction, because $y$ is loopless. 
\begin{figure}[h]
\centering
\includegraphics[width=0.5\textwidth]{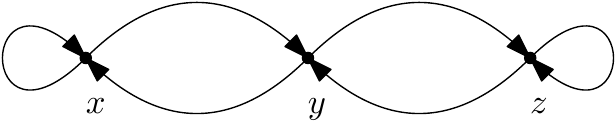}
\caption{A smallest outregular non-semigroup digraph.}\label{fig:notcayleysemigroup}
\end{figure}



Moreover, there is no $k$-outregular non-semigroup digraph with less than three vertices. Since all $0$-outregular digraphs and all $1$-outregular connected digraphs are monoid digraphs (by Theorem~\ref{Monoid_Zelinka}), we only have to worry about $1$-outregular disconnected digraphs and $2$-outregular digraphs. But with less than three vertices there is only one of each kind: the digraph formed by two vertices with a loop each, and the complete digraph with loops of two vertices; and they are respectively isomorphic to $\Cay(\mathbb Z_2,\{0\})$ and $\Cay(\mathbb Z_2,\{0,1\})$, where $\mathbb Z_2$ is the additive group on the integers modulo $2$. 
\end{proof}

\begin{remark}
By examining all cases, one can show that there are in total exactly $3$ outregular non-semigroup digraphs on  $3$ vertices. We omit this discussion here for the sake of brevity. 
\end{remark}

It is known that every graph is isomorphic to some induced subgraph of a Cayley graph of a group~\cite{Babai1978}. Here we show that any sink-free directed graph can be obtained from a monoid digraph by removing a connected component. Note that it is necessary to require the digraph to be sink-free. Similar constructions have been used before (see e.g.~\cite[Theorem 2.2]{gar-20}, also see the standard concept of \emph{transition monoid} in automata theory~\cite[Chapter IV.3]{Pin2012}). Given a graph $G$ (directed or undirected) and a vertex $v$ of $G$, we denote by $G_v$ the connected component of $v$.

\begin{prop}\label{prop:monoid_construction} Let $G=(V,A)$ be a digraph, $\mathcal F$ a set of mappings $V\rightarrow V$ satisfying
\begin{enumerate}[i)]
   \item $\forall f\in\mathcal F\ \ \forall v\in V\ \ (v,f(v))\in A$,
   \item $\forall (v,w)\in A\ \ \exists f\in\mathcal F\ \ f(v)=w$,
\end{enumerate}
and $\langle\mathcal F\rangle$ the monoid generated by $\mathcal{F}$ with respect to composition.
The set $M=V\sqcup\mathcal\langle\mathcal F\rangle$ admits a monoid structure such that $G\cong\Cay(M,\mathcal F)\setminus\Cay(M,\mathcal F)_e$.
\end{prop}

\begin{proof} Consider the following multiplication $\cdot$ on $M$:
\begin{align*}
 \forall x,y\in V \ \,  & x\cdot y=y, \\ 
 \forall x\in V\ \,\forall f\in\langle\mathcal F\rangle \ \, & x\cdot f=f(x), \\
 \forall x\in V\ \,\forall f\in\langle\mathcal F\rangle \ \, & f\cdot x=x, \\
 \forall f,g\in\langle\mathcal F\rangle \ \, & f\cdot g=g\circ f.
\end{align*}
Note that $\textrm{id}\in\langle\mathcal F\rangle$ serves as the neutral element of $\cdot$. To see that $\cdot$ is associative, pick any $\alpha,\beta,\gamma\in M$. If $\gamma\in V$, then $(\alpha\cdot\beta)\cdot\gamma=\gamma=\alpha\cdot\gamma=\alpha\cdot
(\beta\cdot\gamma)$. Thus, let $\gamma\in\langle\mathcal F\rangle$. In the case $\beta\in V$ we have $(\alpha\cdot\beta)\cdot\gamma=\gamma(\beta)=\alpha\cdot\gamma(\beta)=\alpha\cdot
(\beta\cdot\gamma)$; in the case $\alpha\in V$, $\beta\in\langle\mathcal F\rangle$ we have $(\alpha\cdot\beta)\cdot\gamma=\gamma(\beta(\alpha))=\alpha\cdot(\gamma\circ\beta)=
\alpha\cdot(\beta\cdot\gamma)$; and the case $\alpha,\beta\in\mathcal\langle\mathcal F\rangle$ is trivial. Let $H=\Cay(M,\mathcal F)$. It is clear that the subgraph $H[\langle\mathcal F\rangle]$ induced by $\langle\mathcal F\rangle$ is a connected component of $H$, and it coincides with $H_{\textrm{id}}$. The subgraph $H[V]$ induced by $V$ is formed by the rest of the connected components, and by the hypotheses on $\mathcal F$, it is isomorphic to $G$.
\end{proof}

Proposition~\ref{prop:monoid_construction} has the last result of this section as a corollary.

\begin{corol}\label{k-outreg-corol} Let $G$ be a directed graph in which each vertex has at most $k$ out-neighbors, and at least $1$. Then there is a monoid $M$ and $C\subseteq M$ with $|C|=k$ such that $G\cong\Cay(M,C)\setminus\Cay(M,C)_e$.
\end{corol}
\begin{proof}
One can greedily cover the arcs of $G$ by $k$ spanning $1$-outregular subdigraphs: to construct the $i$th such digraph $G_i$ simply take any vertex $x$ that has not yet outdegree $1$ in $G_i$, add any not yet covered arc $(x,y)$ to $G_i$ or an already covered one $(x,y)$ if all are covered, and iterate.

Now each $G_i$ yields a function $f_i$ that just maps every vertex $x$ to its out-neighbor $y$ in $G_i$. The resulting set $\mathcal{F}$ has order $k$ and Proposition~\ref{prop:monoid_construction} yields the result.
\end{proof}

\section{Monoid graphs}\label{section:monoid_graphs}

Until here we have examined directed graphs, but now we forget about orientations and loops: we concentrate on the underlying simple undirected graphs. Also, we shift our attention from semigroups in general to just monoids. The aim of this section is to give several examples of both monoid and non-monoid graphs.

Before starting it is worth noting that in this particular setting some assumptions about the connection set can be made. Let $M$ be a monoid and $C\subseteq M$. First, the graph $\underline{\Cay}(M,C)$ does not depend on the pertinence of $e$ to $C$, so it can be assumed that $e\notin C$. 
Second, since now the direction of the arcs is not important, $C$ can be assumed to be closed under taking existing left-inverses. This is:

\begin{lemma}\label{lem:left-inv} Let $M$ be a monoid, $C\subseteq M$ and $G=\underline{\Cay}(M,C)$. Then $G=\underline{\Cay}(M,N(e))$, where $N(e)$ is the set of neighbors of the neutral element.
\end{lemma}

\begin{proof}
Since $C\subseteq N(e)$ all we have to show is that every edge $\{v,w\}$ of $\underline{\Cay}(M,N(e))$ is an edge of $G$. Suppose that $w=vx$ for some $x\in N(e)\setminus C$. Then $xc=e$ for some $c\in C$. This implies that $\{v,w\}$ is also an edge of $G$, because $wc=vxc=v$.
\end{proof}

We begin with some positive results. A graph $G=(V,E)$ is called a \emph{threshold graph} if there exist non-negative reals $t$ and $w_v$, $v\in V$, such that for every $U\subseteq V$ we have that $\sum_{v\in U} w_v\leq t$ if and only if $U$ is an independent set. Equivalently, a threshold graph is a graph that can be obtained from the one-vertex graph by repeatedly adding an isolated vertex or a vertex which is connected to all the others. More equivalent definitions can be found in \cite[Theorem 1.2.4]{Mah-95}. As a consequence of Proposition~\ref{prop:pretresh}, these graphs are monoid graphs.

\begin{prop}\label{prop:pretresh} Let $G=(V,E)$ be a monoid graph and let $x\notin V$. Then both $G'=(V\cup\{x\},E)$ and $G''=(V\cup\{x\},E\cup\bigcup_{v\in V}\{v,x\})$ are monoid graphs.
\end{prop}
\begin{proof}
Let $M$ be a monoid and $C\subseteq M$ with $G=\underline{\Cay}(M,C)$. Define a monoid structure on $M'=V\cup\{x\}$ using the multiplication from $M$ whenever possible, and defining $xv=vx=x^2=x$ for every $v\in V$. This new multiplication preserves the neutral element from $M$, and it is also associative: let $u,v,w\in M'$ and suppose than at least one of them is equal to $x$; then, $(uv)w=x=u(vw)$. It is straightforward to check that $G'=\underline{\Cay}(M',C)$ and $G''=\underline{\Cay}(M',C\cup\{x\})$.
\end{proof}

\begin{corol}\label{corol:threshold} Threshold graphs are monoid graphs.
\end{corol}

Another class of monoid graphs is the class of forests. This is a corollary to the monoid version of Zelinka's theorem (Theorem~\ref{Monoid_Zelinka}).

\begin{corol}\label{corol:forests} For every forest $F$ there is a monoid $M$ and $a\in M$ such that $F$ is isomorphic to $\underline{\Cay}(M,\{a\})$.
\end{corol}

\begin{proof} Let $F$ be any forest. Pick in each component $\mathcal{C}$ an arbitrary vertex $v$. Direct all edges of $\mathcal{C}$ towards $v$ and add a loop at $v$. Clearly, the resulting digraph satisfies the hypothesis of Theorem~\ref{Monoid_Zelinka}: every component $\mathcal C$ has $z(\mathcal{C})=1$, and some of them maximize $\ell(\mathcal{C})$. 
%
%
%
\end{proof}

Corollary~\ref{corol:forests} leads to the idea of covering a graph with several forests in order to find a monoid representation. More precisely, a \emph{pseudoforest} is a graph having at most one cycle per connected component. The \emph{pseudoarboricity} $p(G)$ (resp. \emph{arboricity} $a(G)$) of a graph $G$ is the minimum number of pseudoforests (resp. forests) needed to cover all the edges of $G$. These invariants can be computed in polynomial time~\cite{Gabow1992}, and results from \cite{NW1964,Frank1976,Picard2006} yield the following 
formulas:
\begin{align*}
a(G)&= \underset{S\subseteq V(G)}{\max} \left\lceil\frac{|E(G[S])|}{|S|-1}\right\rceil\\
p(G)&=  \underset{S\subseteq V(G)}{\max} \left\lceil\frac{|E(G[S])|}{|S|}\right\rceil\!,
\end{align*}
where $E(G[S])$ is the set of edges of the subgraph of $G$ induced by $S$. Not too surprisingly, in the undirected setting the pseudoarboricity has a similar role to that of $k$-outregularity in Section~\ref{sec:digraphs}. The pseudoarboricity of a semigroup graph is bounded by the size of the connection set in any semigroup representation. More precisely:
\begin{obs}
 If $G=\underline{\Cay}(S,C)$ is a semigroup graph, then $p(G)\leq |C|$.
\end{obs}

  Corollary~\ref{k-outreg-corol} together with the fact that $p(G)$ is the minimum, over all orientations $G'$ of $G$, of the maximum outdegree of $G'$, gives:

\begin{corol} Let $G$ be a graph with pseudoarboricity $k$. Then there is a monoid $M$ and $C\subseteq M$ with $|C|=k$ such that $G\cong\underline\Cay(M,C)\setminus\underline \Cay(M,C)_e$.
\end{corol}


Corollary~\ref{corol:forests} can be restated by saying that graphs of arboricity $1$ are monoid graphs. One could ask what happens with higher arboricities. In Proposition~\ref{prop:planar_nonmonoid} we show an infinite family of planar non-monoid graphs of arboricity $2$ and treewidth $3$.

\begin{prop}\label{prop:planar_nonmonoid} The disjoint union $K_4\cup C_{\ell}$ is non-monoid, for all $\ell>1$ with $2,3\nmid\ell$.
\end{prop}

\begin{proof}
Assume for a contradiction that $\underline{\Cay}(M,C)=K_4\cup C_{\ell}$ for some monoid $M$ and $C\subseteq M$. Denote by $\overrightarrow{K_4}$, $\overrightarrow{C_{\ell}}$ the corresponding components of $\Cay(M,C)$, that are directed graphs with possible loops and anti-parallel arcs. Since $\ell$ is odd, there are two consecutive arcs in $\overrightarrow{C_{\ell}}$. Hence, there is a vertex $u$ and $c,c'\in C$ such that the vertices $u,uc,ucc'$ are all different. If the neutral element $e$ is in $K_4$, then $e,cc'$ are either neighbors or the same vertex. But then $u,ucc'$ should also be neighbors or the same vertex, a contradiction. 

Therefore $e$ must be in $C_{\ell}$. Hence by Lemma~\ref{lem:left-inv} we can suppose that $C=\{c,c'\}$ with $c\neq c'$. The periods of $c,c'$ can only be $1$, $2$ or $\ell$, so by Theorem~\ref{Monoid_Zelinka} both $\underline{\Cay}(M,\{c\})$ and $\underline{\Cay}(M,\{c'\})$ are pseudoforests without cycles of length $3$ or $4$. Then each element of $C$ corresponds to at most $3$ non-loop arcs of $\overrightarrow{K_4}$, and it must correspond to exactly $3$, because $K_4$ has $6$ edges; moreover, none of the edges of $K_4$ correspond to both $c,c'$. Therefore, if we group the edges of $K_4$ depending on whether they come from $c$ or from $c'$, we obtain two edge-disjoint copies of the path $P_4$. We now distinguish two cases. 

If there is an edge in $C_{\ell}$ corresponding to both $c,c'$, any endomorphism mapping $e$ to a vertex of $K_4$ implies the existence of a loop in $\overrightarrow{K_4}$. Lemma~\ref{lem:monosabi} guarantees the existence of such an endomorphism. 

If otherwise every edge in $C_{\ell}$ corresponds to either $c$ or $c'$, since $\ell$ is odd then there is a loop in $\overrightarrow{C_{\ell}}$. Thus, by Lemma~\ref{lem:monosabi}   there is a loop in $\overrightarrow{K_4}$. We can suppose that this loop corresponds to $c$. Let $x$ be the vertex of the loop. Wherever $x$ lies along the $c$-copy of $P_4$, it implies there are vertices $v\neq w$, different from $x$, such that $vc=w,wc=x$. This implies that $e,c,c^2$ are all different. We know that $cc'\in\{e,c,c^2\}$, so $xc'=(xc)c'=x(cc')=x$, and there cannot be more loops in $\overrightarrow{K_4}$. Now, $wc'=vcc'\in\{v,w,x\}$, but that is a contradiction.
\end{proof}

\begin{figure}[h]
\centering
\includegraphics[width=.5\textwidth]{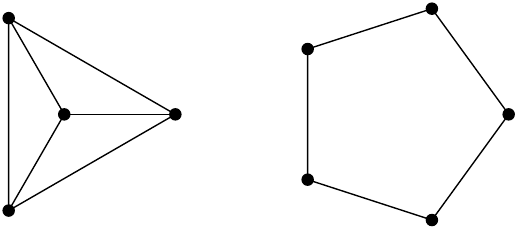}
\caption{The graph $K_4\cup C_5$ is the smallest non-monoid graph we know of.}

\end{figure}\label{fig:K4C5}

\begin{remark}
We have checked by computer, that all graphs on up to $6$ vertices are monoid graphs. The instance of Proposition~\ref{prop:planar_nonmonoid} on $9$ vertices depicted in Figure~\ref{fig:K4C5} is the smallest non-monoid graph we know of. Is it the smallest one?
\end{remark}

The construction and the argument in Proposition~\ref{prop:planar_nonmonoid} are based on the fact that the graph is disconnected. We dedicate the remainder of the section to the construction of $k$-connected non-monoid graphs for arbitrarily large $k$ (Theorem~\ref{thm:connected}). 

We start with some more technical definitions and lemmas. Given an undirected graph $G=(V,E)$, consider the set $\mathcal F=\{f:V\rightarrow E\ |\ \forall v\in V: v\in f(v)\}$ and define $\beta(G,k)=\underset{f_1,...,f_k\in\mathcal F}{\max}\;\alpha(\,(V,E\setminus\bigcup_{i=1}^k f_i(V))\,)$, where $\alpha$ denotes the independence number.

\begin{lema}\label{lem:necesary}
Let $G=(V,E)$ be a triangle-free $n$-vertex graph with minimum degree $\delta$ and maximum degree $\Delta\geq 2$. If for $k\geq 0$, $\ell>2\Delta+2k+1$ there is a monoid $M$ and $C\subseteq M$ such that $\underline{\Cay}(M,C)$ consists of $G$ and $K_{\ell}$ joined by $k$ edges, then $\beta(G,k)\geq\frac{n}{\Delta-1}\left(\frac{\delta}{2}-k-1\right)$.
\end{lema}
\begin{proof}
By Lemma~\ref{lem:left-inv}, we can suppose that the elements of $C$ are exactly the neighbors of $e$. Hence, $e$ has to be a vertex of $K_{\ell}$: otherwise $|C|\leq\Delta+k$, but this would imply that there are at most $\ell(\Delta+k)+k<\frac{\ell(\ell-1)}{2}+k$ edges incident to vertices of $K_{\ell}$.

Let $E'$ be the set of edges joining $K_{\ell}$ and $G$, and $C'=\{c\in C\mid \exists d\in C\ cd\notin C\cup\{e\}\}$. For each $c\in C'$ let $\varepsilon_c$ be the edge $\{e,c\}\in E'$ if $c\in V$, or any edge of $E'$ adjacent to $c$ if $c\notin V$. The map $C'\rightarrow E'$ defined by $c\mapsto\varepsilon_c$ is injective, so $|C'|\leq |E'|=k$.

For each $u\in V$ define 
\[\textrm{indeg}^{C'}(u)=\left|\,\{v\in V\setminus\{u\} \mid \exists c\in C\setminus C'\ vc=u\}\,\right|,\]
\[\textrm{outdeg}^{C'}(u)=\left|\,\{v\in V\setminus \{u\} \mid \exists c\in C\setminus C'\ uc=v\}\,\right|.\]
Consider the set $O=\{u\in V\,|\,\textrm{outdeg}^{C'}(u)=0\}$. Let $v\in V\setminus O$ and $w\in V\setminus\{v\}$ with $vc=w$ for some $c\in C\setminus C'$. Suppose that there are two different $u_1,u_2\in V\setminus\{v\}$ and $c_1,c_2\in C\setminus C'$ with $u_1 c_1 = u_2 c_2=v$. We can suppose that $u_1\neq w$. Observe that $c_1 c\in C\cup\{e\}$, because otherwise $c_1\in C'$. In fact, since $u_1,v,w$ are pairwise different, $c_1 c\in C$, but then these vertices form a triangle, a contradiction. Therefore $\textrm{indeg}^{C'}(v)\leq 1$. Now we can write
\[\frac{n\delta}{2}-nk\,\leq\, |E|-n|C'|\,\leq\, \sum_{u\in V}\textrm{indeg}^{C'}(u)\,=\,\sum_{u\in O}\textrm{indeg}^{C'}(u)+\sum_{u\in V\setminus O}\textrm{indeg}^{C'}(u)\,\leq\]
\[\leq\,\Delta|O|+|V\setminus O|\,=\, (\Delta-1)|O|+n.\]

Note that $O$ is an independent set of the graph $(V,E\setminus\bigcup_{c\in C'}f_c(V))$, where $f_c:V\rightarrow E$ is defined by $f_c(u)=\{u,uc\}$. This yields the desired result.
\end{proof}

For $\lambda>0$ a graph $G$ is an \emph{$(n,d,\lambda)$-graph} if $G$ has $n$ vertices, is $d$-regular, and all the Eigenvalues of its adjacency matrix are in $[-\lambda,\lambda]\cup\{d,-d\}$. The expander mixing lemma (see, e.g.,~\cite{alo-88,alo-16},~\cite[Section 5]{Haemers} or~\cite[Lemma 2.1]{anorak}) asserts that if $S, T$ are two sets of vertices in such a graph then
\[ \left |e(S,T) - d \frac{|S||T|}{n} \right| \leq \lambda \sqrt{|S||T|\left(1 - \frac{|S|}{n}\right) \left(1 - \frac{|T|}{n} \right)}, \]
where $e(S, T)$ is the number of (ordered) edges $uv$ with $u \in S, v \in T$.

A standard application of the expander mixing lemma is bounding from above the independence number $\alpha(G)$. Here we show that $\beta(G,k)$ can be bounded in a similar manner.

\begin{lemma}\label{lem:beta}
If $G=(V,E)$ is an $(n,d,\lambda)$-graph and $d\geq 1$, then $\beta(G,k)\leq\frac{n}{d}(\lambda+2k)$ for all $k\geq 0$.
\end{lemma}
\begin{proof}
Let $f_1,...,f_{k}\in\mathcal F=\{f:V\rightarrow E\ |\ \forall v\in V\ v\in f(v)\}$ and let $S$ be an independent set of $(V,E\setminus\bigcup_{i=1}^{k} f_i(V))$. The number of edges of the subgraph of $G$ induced by $S$ is at most $k|S|$. The expander mixing lemma yields 
\[|S|\leq\frac{n}{d}(\lambda+2k),\]
and this implies the result.
\end{proof}

Our strategy will be to confront the bounds of Lemma~\ref{lem:necesary} and Lemma~\ref{lem:beta} by using $(n,d,\lambda)$-graphs in the construction of the examples. For this, we will also need some control over the connectivity of $(n,d,\lambda)$-graphs.

\begin{lemma}\label{lem:k}
If $G=(V,E)$ is an $(n,d,\lambda)$-graph, $d\geq 1$ and $0\leq k<\frac{(d-\lambda)^2}{d}+1$, then $G$ is $k$-connected.
 \end{lemma}

\begin{proof}
Suppose that $G$ is not $k$-connected. 
Since $k<\frac{(d-\lambda)^2}{d}+1<d+1\leq n$ (it can be always assumed that $\lambda\leq d$), we know that there exist non-empty $S,T,U\subseteq V$ with $S\dotcup T\dotcup U=V$, $|U|=k-1$, $e(S,T)=0$ and $|S|\leq|T|$.

Applying the expander-mixing lemma to $S,T\dotcup U$ we get that
\[\frac{(d-\lambda)|S||T\dotcup U|}{n}\leq e(S,T\dotcup U).\]
On the other hand, $e(S,T\dotcup U)=e(S,U)\leq |S|(k-1)$, so
\[\frac{(d-\lambda)(n-|S|)}{n}\leq k-1.\]
Finally, we must have that $|S|\leq\frac{n\lambda}{d}$; otherwise, applying the expander mixing lemma to $S,T$ would lead to a contradiction. This yields
\[\frac{(d-\lambda)^2}{d}+1\leq k.\]
\end{proof}

We are ready to prove the main theorem of this section.

\begin{theorem}\label{thm:connected} For every non-negative integer $k$ there exists a $k$-connected graph that is not a monoid graph. 
\end{theorem}
\begin{proof}
The family of (triangle-free) non-bipartite Ramanujan graphs of unbounded degree given by Lubotzky, Phillips and Sarnak~\cite{lub-88} consists of $(n,d,\lambda)$-graphs with $\lambda=2\sqrt{d-1}$. Let $G$ be one of such graphs, also satisfying $d>(\sqrt{6k+6}+2)^2=6k+4\sqrt{6k+6}+10$, and let $\ell>2d+2k+1$. Choose $k$ different vertices $u_1,...,u_k$ of $G$ and $k$ different vertices $v_1,...,v_k$ of $K_{\ell}$. Consider the graph $H$ consisting in $G$, $K_{\ell}$ and the $k$ additional edges $\{u_1,v_1\},...,\{u_k,v_k\}$.

Observe that the inequality $(d-2\sqrt{d-1})^2>(d-2\sqrt d)^2$ holds for $d>8$. Hence we have $\frac{(d-\lambda)^2}{d}+1>\frac{(d-2\sqrt{d})^2}{d}+1=(\sqrt{d}-2)^2+1>6k+7>k$. By Lemma~\ref{lem:k}, $G$ is $k$-connected, so $H$ is $k$-connected. 

Now suppose that $H$ is a monoid graph. Lemma~\ref{lem:necesary} and Lemma~\ref{lem:beta} give
\begin{align*}
\frac{n}{d}\left(\frac{d}{2}-k-1\right)\leq\beta(G,k)&\leq\frac{n}{d}(\lambda+2k) \\
d-2\lambda-6k-2&\leq 0 \\
d-4\sqrt d-6k-2&\leq 0,
\end{align*}
but the set of solutions of the inequality $x^2-4x-6k-2\leq 0$ is $\left[2-\sqrt{6k+6},\,2+\sqrt{6k+6}\,\right]$, so the above inequalities do not hold. This contradiction implies that $H$ is not a monoid graph.
\end{proof}

\section{Generated monoid trees}\label{sec:monoid_tree}


In the present section we investigate the more restricted class of monoid graphs, where the connection set generates the monoid. This restriction has received some attention with respect to semigroups, see~\cite[Chapter 11.3]{Kna-19}. More precisely, we study the following  problem: given a tree $T$, determine whether there exist a monoid $M$ and $C$ with $M=\langle C\rangle$ such that $T=\underline{\Cay}(M,C)$. 

We know already, that trees are monoid graphs by Corollary~\ref{corol:forests}. So the present section is an investigation on how much more restrictive the condition $M=\langle C\rangle$ is.
In particular, we provide a sufficient condition (Theorem~\ref{positive_tree}) and a necessary condition (Theorem~\ref{negative_tree}) for a tree to be such a \emph{generated monoid graph}, both stated in similar terms. However, a characterization remains open.  

We start by introducing some notation borrowed from the context of distance-regular graphs~\cite[Section 20]{big-93}. Let $G=(V,E)$ be a connected graph and $e$ a fixed vertex of $G$. For every vertex $x$ we define the number of \emph{successors of $x$} with respect to $e$ as  $b_e(x)=\left|\{y\in V\ |\ \{x,y\}\in E,\ d(e,y)=d(e,x)+1\}\right|$. To avoid unnecessary complications, throughout the section we will assume that in an undirected Cayley graph $\underline{\Cay}(M,C)$ the neutral element is not in $C$.

\begin{theorem}\label{positive_tree} Let $T=(V,E)$ be a tree and $e\in V$. If for every $x,y\in V$ with $d(e,x)>d(e,y)$ we have $b_e(x)\leq b_e(y)$, then $T$ is a generated monoid graph.
\end{theorem}

\begin{proof}
Let $C=\{c_1,...,c_n\}$ the set of neighbors of $e$. Let $\Sigma=\{\overline c_1,...,\overline c_n\}$ be a set of $n$ different symbols. We construct a multi-digraph with $\Sigma$-colored arcs $T'=(V,A)$ which has $T$ as its underlying simple undirected graph. Given a vertex $x\in V$, call $x_1,...,x_{b_e(x)}$ its successors respect to $e$; the successor $x_{b_e(x)}$ will be also denoted by $x_i$ for $i=b_e(x)+1,...,n$. If $x$ has no successor, then use the notation $x=x_1=...=x_n$. The arcs of $T'$ of color $\overline c_i$ are precisely $(x,x_i)$ for each $x\in V$. $T'$ depends on the choice made when labeling the successors, but in what follows that will not matter. 


Now, for each vertex $x$ and each finite word $w\in\Sigma^*$ from the alphabet $\Sigma$, define $x*w$ to be the ending vertex of the directed walk starting at $x$ and defined by $w$. Observe that for any $w\in\Sigma^*$ $d(e,e*w)\leq\textrm{length}(w)$, which is an equality if $e*w$ is not a leaf of $T$. Let $\sim$ be the equivalence relation in $\Sigma^*$ defined by $w\sim w'\Leftrightarrow e*w=e*w'$. Construct a third graph $T''$, also directed and with multiple $\Sigma$-colored arcs, but this time with vertex set $\Sigma^*/\!\sim$: there is an arc of color $\overline c$ from $[w_1]$ to $[w_2]$ if and only if $[w_1\overline c]=[w_2]$. Note that the map from $\Sigma^*/\!\sim$ to $V$ sending $[w]$ to $e*w$ is an isomorphism between $T''$ and $T'$.

Endow $\Sigma^*/\!\sim$ with the following multiplication: $[w_1][w_2]=[w_1 w_2]$ for each $w_1,w_2\in\Sigma^*$.  The following will be useful.
\begin{claim}
 For any $w,w'\in\Sigma^*$ we have $b_e(e*(ww'))\leq b_e(e*w')$.
\end{claim}
\begin{proof}
In the case that $d(e,e*(ww'))>d(e,e*w')$, this is implied by the assumed hypothesis. And in the opposite case, $d(e,e*(ww'))\leq d(e,e*w')\leq\textrm{length}(w')\leq\textrm{length}(ww')$, so either there is a chain of equalities and $w$ is the empty word (in which case we are done), or $e*(ww')$ is a leaf, and then $b_e(e*(ww'))=0$.
\end{proof}

We have to check that the definition is independent of the representatives.
Let $w_1,w_2,w'_1,w'_2\in\Sigma^*$ with $w_1\sim w'_1,w_2\sim w'_2$. Since $e*(w_1 w_2)=(e*w_1)*w_2=(e*w'_1)*w_2=e*(w'_1 w_2)$, we only have to bother if $w_2\neq w'_2$. We distinguish two cases:

If $\textrm{length}(w_2)=\textrm{length}(w'_2)$, write $w_2=\overline d_1 ...\overline d_r$ and $w'_2=\overline d'_1...\overline d'_r$ with $\overline d_i,\overline d'_i\in\Sigma$ for $1\leq i\leq r$. We have that, for each $i$, either $\overline d_i=\overline d'_i$ or $b_e(e*(\overline d_1...\overline d_{i-1}))\leq j,j'$, where $\overline d_i=\overline c_j$ and $\overline d'_i=\overline c_{j'}$. We now use the Claim
to conclude that $b_e(e*(w_1\overline d_1...\overline d_{i-1}))\leq b_e(e*(\overline d_1...\overline d_{i-1}))$ for each $i$. This implies that $w_1\overline d_1...\overline d_r\sim w_1\overline d'_1...\overline d'_r$. 

If $\textrm{length}(w_2)<\textrm{length}(w'_2)$, write $w'_2=w''_2 w'''_2$ with $w''_2,w'''_2\in\Sigma^*$ and $\textrm{length}(w''_2)=\textrm{length}(w_2)$. We know that $e*w_2=e*w'_2$ is a leaf, so $e*(w_2 w'''_2)=e*w'_2$. Here $\textrm{length}(w_2 w'''_2)=\textrm{length}(w'_2)$, so $e*w_2=e*w''_2$, which can be obtained following the argumentation of the first case. Again by the first case, $e*(w_1 w_2)=e*(w_1 w''_2)$, and by the Claim, this vertex is a leaf of $T$. Hence, $e*(w_1 w_2)=(e*(w_1 w''_2))* w'''_2=e*(w_1 w'_2)$, and we are done.

The defined multiplication is associative; showing it is an easy task: $([w_1][w_2])[w_3]=[(w_1 w_2) w_3]=[w_1 (w_2 w_3)]=[w_1]([w_2][w_3])$. Moreover, it has a neutral element, the $\sim$-class of the empty word. Therefore $\Sigma^*/\!\sim$ with this operation is a monoid, and clearly $T''\cong\Cay_{\mathrm{col}}(\Sigma^*/\!\sim,\Sigma/\!\sim)$. 
\end{proof}

Let $T=(V,E)$ be a tree and $e\in V$. For each $c$ neighbor of $e$ we define the \emph{branch} of $c$ respect to $e$ as $\mathcal B_e(c)=\{x\in V\mid d(c,x)<d(e,x)\}$. If $x\in\mathcal B_e(c)$, then we also say that $\mathcal B_e(c)$ is the branch of $x$ respect to $e$. The branches we will consider in generated monoid trees will be only with respect to (possible) neutral elements.

Given a monoid $M$ and $x\in M$, let $\varphi_x:M\rightarrow M$ be the map corresponding to left-multiplication by $x$.

\begin{prop}\label{negative_tree} Let $T=(V,E)$ be a generated monoid tree, i.e., 
$T=\underline{\Cay}(M,C)$ for some monoid $M$ and some $C$ with $M=\langle C\rangle$. Then for each $x\in V\setminus\{e\}$ there is a $y\in V$ with $d(e,x)=d(e,y)+1$ and $b_e(x)\leq b_e(y)$. Moreover, if $\Aut(T)\cap\{\varphi_x\mid x\in M\}=\{\mathrm{id}\}$, then for each vertex $x$ and each $c\in C$ there is a vertex $y_c\in\mathcal B_e(c)$ with $d(e,y_c)\leq d(e,x)+1$ and $b_e(y_c)\leq b_e(x)+\varepsilon$, where $\varepsilon=\begin{cases}0 \text{ if } x\in C\cup\{e\}, \\ 1 \text{ otherwise.}\end{cases}$
\end{prop}

\begin{proof} 
We start with a useful fact.
\begin{claim}
 For each $c\in C$ and each $x\in M$ we have $d(e,cx)\leq d(e,x)+1$.
\end{claim}
\begin{proof}
Let $e=x_0,x_1,...,x_r=x$ be the successive vertices of a shortest path between $e$ and $x$, where $r=d(e,x)$. Then the vertices $e,c,cx_1,...,cx_r=cx$, maybe after adding some loops, define a walk between $e$ and $cx$. Therefore $d(e,cx)\leq d(e,x)+1$.
\end{proof}

Given any $x\in M\setminus\{e\}$, let $c\in C$ with $x\in\mathcal B_e(c)$. Write $x=c_1...c_r$ with $c_i\in C$, $r=d(e,x)$ and $c_1=c$. Clearly $y=c_2...c_r$ satisfies $d(e,y)+1\leq r$. By the Claim this is, in fact, an equality. For an element $v\in M$ let $\textrm{o}(v)=|vC\setminus\{v\}|$. Observe that $\textrm o(v)=\begin{cases}b_e(v) &\text{ if } \forall c'\in C\ d(e,vc')=d(e,v)+1\\ b_e(v)+1 &\text{ otherwise} \end{cases}$ and $\textrm o(cv)\leq\textrm o(v)$. Hence to show that $b_e(x)\leq b_e(y)$ it will be enough to see that if $\textrm{o}(y)=b_e(y)+1$ then $\textrm{o}(x)=b_e(x)+1$. Now, $\textrm{o}(y)=b_e(y)+1$ means that there is some $z\in M$ and some $c''\in C$ with $yc''=z$ and $d(e,z)=d(e,y)-1$. Again by the Claim, $d(e,cz)\leq d(e,z)+1=d(e,y)=d(e,x)-1$. Since $xc''=cz$, we have that $\textrm{o}(x)=b_e(x)+1$. 

For the second part, let $y_c=cx$. Suppose that $y_c\notin\mathcal B_e(c)$. Take a walk between $e$ and $x$, and let $e=x_0,x_1,...,x_r=x$ be its successive vertices. The vertices $c=cx_0,cx_1,...,cx_r=cx$, maybe after adding some loops, define a walk between $c$ and $cx$. Therefore $cx_i=e$ for some $i$. Since $M$ is finite, $\varphi_c$ is an automorphism, contradicting that $\Aut(T)\cap\{\varphi_x\mid x\in M\}=\{\mathrm{id}\}$. Therefore, $y_c\in\mathcal B_e(c)$. The fact that $b_e(y_c)\leq b_e(x)+1$ is a consequence of $\textrm o(cx)\leq\textrm o(x)$. Furthermore, if $x\in C$ then $\textrm{o}(x)=b_e(x)$, because otherwise $\varphi_x\in\Aut(T)$, so $b_e(y_c)\leq b_e(x)$ in this case. If $x=e$ the same conclusion holds, because $\textrm{o}(x)=b_e(x)$ directly. Finally, what is left follows from the Claim.
\end{proof}

The following lemma shows that the condition for the second part of Proposition~\ref{negative_tree} is not as restrictive as it may seem.

\begin{lemma}\label{tree_symmetries} Let $T=(V,E)$ be a tree satisfying $T=\underline{\Cay}(M,C)$ for some monoid $M$ and some $C$ with $M=\langle C\rangle$. If $\Aut(T)\cap\{\varphi_x\mid x\in M\}\neq\{\mathrm{id}\}$ then there is some $c\in C$ such that $c^2=e$ and $\Aut(T)\cap\{\varphi_x\mid x\in M\}=\{\mathrm{id,\varphi_c}\}$.

\end{lemma}

\begin{proof}
Let $x\in M\setminus\{e\}$ such that $\varphi_x\in\Aut(T)$. Write $x=c_1...c_r$ with $c_i\in C$. Since $M$ is finite, $\varphi_{c_1},...,\varphi_{c_r}\in\Aut(T)$. This implies that $c_1^2=...=c_r^2=e$, or else $T$ would have a cycle. Now let $c,c'\in C$ with $\varphi_c,\varphi_{c'}\in\Aut(T)$. If $cc'\neq e$ then there is a cycle in $T$. Indeed, let $k$ be the order of $cc'$. It will be enough to show that the vertices $e,c,cc',cc'c,...,(cc')^{k-1},(cc')^{k-1}c$ are all different. It is clear that $e,cc',...,(cc')^{k-1}$ are pairwise different, and since right-multiplication by $c$ has to be bijective, so are $c,(cc')c,...,(cc')^{k-1}c$. Moreover there is no overlap between these two sets of vertices: the distance from $e$ is even in the first case and odd in the second (the presence of consecutive vertices in the succession $e,c,cc',cc'c,...,(cc')^{k-1},(cc')^{k-1}c$ would imply that $c=e$ or $c'=e$ after the cancellation of the left terms). Therefore, $cc'=e$, which means that $c=c'$.
\end{proof}

A drawback of Proposition~\ref{negative_tree} is that it can be difficult to use without knowing the (possible) location(s) of the neutral element. Luckily, the number of candidates can be sometimes restricted.

\begin{lemma}\label{localize_neutral} Let $T=(V,E)$ be a generated monoid tree, i.e., 
$T=\underline{\Cay}(M,C)$ for some monoid $M$ and some $C$ with $M=\langle C\rangle$ and let $\Delta$ be the maximum degree of $T$. We have $\Delta-1\leq\deg(e)\leq\Delta$. 
\end{lemma}

\begin{proof}
Let $u\neq e$ be a vertex of $T$. Suppose that in $\Cay(M,C)$ there are two arcs $(x_1,u),(x_2,u)$ incident to $u$, and that the arcs $(u,x_1),(u,x_2)$ do not exist. We know that in $T$ any vertex $v$ is reachable from $e$ by a unique shortest path $P(v)$. The edge $\{x_1,u\}$ is the last one of either $P(x_1)$ or $P(u)$. Since the corresponding paths in $\Cay(M,C)$ are directed, from $e$ to the target vertex, $\{x_1,u\}$ is the last edge of $P(u)$. By the same argument, so is $\{x_2,u\}$, and this shows that in $\Cay(M,C)$ vertex $u$ has at most one in-neighbor which is not an out-neighbor. Therefore $\deg(u)\leq |C|+1$, and the fact that $\deg(e)\geq |C|$ yields the conclusion.
\end{proof}

We apply the above results to present two families of trees, one being generated monoid graphs and the other not  being generated monoid graphs.
A \emph{perfect $k$-ary tree} of height $h$, $T_{k,h}$ is a rooted tree in which the root has $k$ neighbors, every non-leaf vertex has also $k$ successors, i.e., neighbors that are further from the root than itself, and all the leaves are at distance exactly $h$ from the root. See Figure~\ref{fig:3-tree} for an illustration. A direct consequence of Proposition~\ref{positive_tree} is: 
\begin{obs}
Every perfect $k$-ary tree $T_{k,h}$ is a generated monoid graph.
\end{obs}

\begin{figure}[h]
\centering
\includegraphics[scale=0.5]{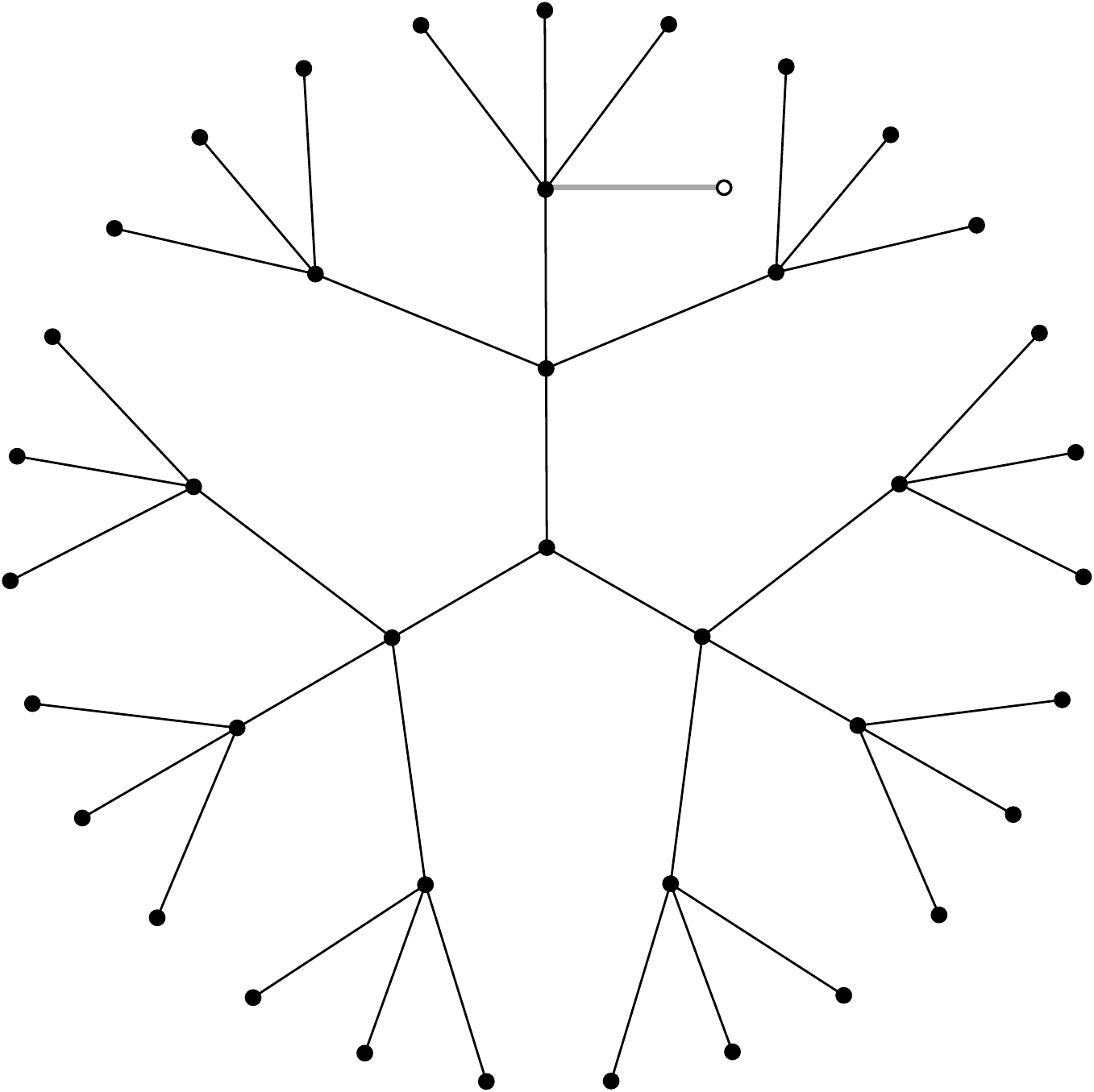}
\caption{The perfect $3$-ary tree of height $3$, and (with the grey edge added) the tree $T^+_{3,3}$.}\label{fig:3-tree}
\end{figure}


Denote by $T^+_{k,h}$ the tree obtained from $T_{k,h}$ by adding an extra leaf at distance $h$ to the root.  The tree $T^+_{3,3}$ is depicted in Figure~\ref{fig:3-tree}.
\begin{prop}
 For $k\geq 3$ and $h\geq 2$, the tree $T^+_{k,h}$ is not a generated monoid graph.
\end{prop}
\begin{proof}
 
Suppose that $T^+_{k,h}=\underline{\Cay}(M,C)$ for some monoid $M$ and some $C$ with $M=\langle C\rangle$. By Lemma~\ref{tree_symmetries}, there is no non-trivial element of $M$ corresponding to an automorphism of $T^+_{k,h}$; such an automorphism would map the neutral element to a vertex which is at another depth (distance from the root). So Proposition~\ref{negative_tree} can be fully used. Let $\ell$ be the depth of the neutral element. By Lemma~\ref{localize_neutral}, $1\leq\ell\leq h-1$, because $h\geq 2$. We know that there is a leaf $x$ with $d(e,x)=h-\ell$ (and $b_e(x)=0$), but none of the vertices $y$ in the branch of the root with $d(e,y)\leq h-\ell+1$ is a leaf, and in fact they all satisfy $b_e(y)\geq k-1\geq 2$, contradicting the second part of Proposition~\ref{negative_tree}.
\end{proof}
%

In the last construction, due to the high symmetry, considering different possible placements of the neutral element was easy. If we add two leaves instead of just one to the same vertex, then by Lemma~\ref{localize_neutral} in the resulting graph  only one placement of the neutral element has to be analyzed.

\begin{remark}
The above results are enough to determine, up to order $7$, if a tree is a generated monoid graph or not. It turns out that all of them are generated monoid graphs, except one (see the left of Figure~\ref{fig:smalltree}). On the other hand on $8$ vertices we can find a first tree for which our results do not determine whether it is a generated monoid tree or not (see the right of Figure~\ref{fig:smalltree}). 
\end{remark}
\begin{figure}[h]
\centering
\includegraphics[width=.65\textwidth]{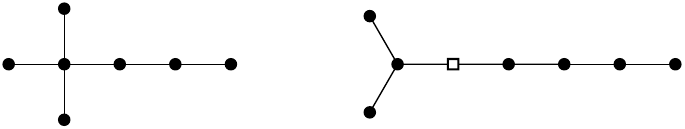}
\caption{Left: the smallest non-generated monoid tree. Right: a tree of order $8$ which is not handled by above results. Still, they yield that the only candidate to neutral element is the squared vertex.}\label{fig:smalltree}
\end{figure}

\section{Concluding remarks}\label{sec:conclusion}
 
We have studied the classes of monoid and semigroup graphs and digraphs. One of our first results is a Sabidussi-type characterization of monoid digraphs (Lemma~\ref{lem:monosabi}). Since from every semigroup $S$ one can obtain a monoid $S^+$ by adjoining a neutral element, this can in a certain way be used to obtain a characterization of semigroup digraphs: a digraph is a semigroup digraph if and only if a source can be added, such that this new vertex plays the role of $e$ in Lemma~\ref{lem:monosabi}. However, we wonder if there is an intrinsic characterization: 
\begin{quest}
Is there a Sabidussi-type characterization of semigroup digraphs?
\end{quest}

%

The central piece of this paper is about monoid graphs. Our knowledge about this class remains superficial and its structure seems hard to capture. Note in particular, that the class of monoid graphs is not closed under disjoint union by Proposition~\ref{prop:planar_nonmonoid}, which would have been a generalization of Proposition~\ref{prop:pretresh}.
The positive examples we know are forests and threshold graphs (Corollaries \ref{corol:threshold} and \ref{corol:forests}).  Forests coincide with the class of graphs of treewidth $1$. On the other hand the non-monoid graphs from Proposition~\ref{prop:planar_nonmonoid} have treewidth $3$. How about graphs of treewidth $2$? E.g., outerplanar graphs or \emph{$2$-trees}: graphs that can be obtained from $K_3$ by successively adding a new vertex with exactly two neighbors who form an edge. Figure~\ref{fig:ExampleMonoidGraph2} depicts an outerplanar $2$-tree that is a monoid graph. 
\begin{quest}
Are $2$-trees monoid graphs?  Are outerplanar graphs monoid graphs?  Are there series-parallel non-monoid graphs?
\end{quest}

Another direction are semigroup graphs; here the fundamental question is open. We formulate it in the more provocative way, even if we believe the answer to be negative as in the case of posets, see~\cite{gar-20}.
\begin{quest}
 Is every graph a semigroup graph?
\end{quest}


Finally, we have studied trees that are generated monoid graphs. We have found necessary and sufficient conditions for a tree to form part of this class (Theorem \ref{positive_tree} and Proposition \ref{negative_tree}). These conditions are in terms of simple parameters related to the distance to the root of a tree $T$. It would be nice to find a characterization in these terms. More generally both conditions are computationally easy to verify. We wonder:
\begin{quest}
 Can it be tested in polynomial time if a tree is a generated monoid tree?
\end{quest} 

\subsubsection*{Acknowledgements} This research was initiated in the BGSMath course \emph{Fragments of algebraic graph theory} 2021. Thanks to participants, lecturers, and organizers. K.K was partially supported by the Spanish \emph{Ministerio de Econom\'ia,
Industria y Competitividad} through grant RYC-2017-22701 and by MICINN through grant PID2019-104844GB-I00.

\newpage

{\small
\bibliography{monoid,lit}
\bibliographystyle{my-siam}
}
\end{document}